\documentclass[12pt]{amsart}
\usepackage{amssymb,latexsym,comment,url}
\usepackage[all]{xy}

\newcommand{\aaa}{\mathfrak{a}}
\newcommand{\bbb}{\mathfrak{b}}
\newcommand{\Aut}{\operatorname{Aut}}
\newcommand{\divv}{\operatorname{div }}
\newcommand{\Image}{\operatorname{Im}}
\newcommand{\Br}{\operatorname{Br}}

\newcommand{\CL}{{\mathcal L}}

\newcommand{\Gal}{\operatorname{Gal}}

\newcommand{\ord}{\operatorname{ord}}
\newcommand{\Ind}{\operatorname{Ind}}
\newcommand{\F}{{\mathbb F}}
\newcommand{\GL}{\operatorname{GL}}

\newcommand{\summ}{\operatorname{sum}}
\newcommand{\Pic}{\operatorname{Pic}}
\newcommand{\Mat}{\operatorname{Mat}}
\newcommand{\Map}{\operatorname{Map}}
\newcommand{\Kbar}{\overline{K}}
\newcommand{\Lbar}{\overline{L}}
\newcommand{\lin}{{\lambda}} 
\newcommand{\linspec}{{\ell}} 
\newcommand{\Qbar}{\overline{\Q}}
\newcommand{\Rbar}{\overline{R}}

\newcommand{\isom}{\cong}
\newcommand{\nr}{{\text{\rm nr}}}

\newcommand{\zeroE}{{0}} 
\newcommand{\OO}{{\mathcal{O}}}
\newcommand{\PP}{{\mathbb P}}
\newcommand{\pp}{{\mathfrak p}}
\newcommand{\qq}{{\mathfrak q}}
\newcommand{\Q}{{\mathbb Q}}

\newcommand{\Res}{{\operatorname{res}}}
\newcommand{\Sset}{{\mathcal S}}
\newcommand{\rank}{\operatorname{rank}}
\newcommand{\ra}{{\longrightarrow}}
\newcommand{\recip}[1]{\tfrac{1}{#1}}
\newcommand{\Tr}{\operatorname{Tr}}
\newcommand{\Ob}{\operatorname{Ob}}

\newcommand{\Z}{{\mathbb Z}}
\newcommand{\mult}{{\times}}

\DeclareMathOperator{\inv}{inv}

\newcommand{\bbQ}{{\Q}}
\newcommand{\bbZ}{{\Z}}

\newcommand{\Sha}{\mbox{\wncyr Sh}}

\newfont{\wncyr}{wncyr10 at 12pt}
\newfont{\wncyrten}{wncyr10 at 10pt}

\newtheorem{proposition}{Proposition}[section]
\newtheorem{theorem}[proposition]{Theorem}
\newtheorem{lemma}[proposition]{Lemma}
\newtheorem{corollary}[proposition]{Corollary}

\theoremstyle{definition}
\newtheorem{definition}[proposition]{Definition}
\newtheorem{Remark}[proposition]{Remark}
\newtheorem{Problem}[proposition]{Problem}

\begin{document}

\title[Computing the Cassels--Tate pairing]%
{Computing the Cassels--Tate pairing on the 3-Selmer group of an elliptic curve}

\author{Tom~Fisher}
\address{University of Cambridge,
         DPMMS, Centre for Mathematical Sciences,
         Wilberforce Road, Cambridge CB3 0WB, UK}
\email{T.A.Fisher@dpmms.cam.ac.uk}

\author{Rachel~Newton}
\address{University of Leiden,
Mathematical Institute,
PO Box 9512,
2300 RA Leiden,
The Netherlands}
\email{newtonrd@math.leidenuniv.nl}

\date{6th June 2013}  

\begin{abstract}
  We extend the method of Cassels for computing the Cassels--Tate
  pairing on the $2$-Selmer group of an elliptic curve, to the case of
  $3$-Selmer groups. This requires significant modifications
  to both the local and global parts of the calculation.  Our method
  is practical in sufficiently small examples, and can be used to
  improve the upper bound for the rank of an elliptic curve obtained
  by $3$-descent.
\end{abstract}

\maketitle

\renewcommand{\baselinestretch}{1.1}
\renewcommand{\arraystretch}{1.3}
\renewcommand{\theenumi}{\roman{enumi}}
\renewcommand{\theenumii}{\alph{enumii}}

\section*{Introduction}

The determination of the Mordell--Weil group $E(K)$ of an elliptic
curve $E$ over a number field $K$ is usually tackled by means of
computing the $n$-Selmer group $S^{(n)}(E/K)$ for some integer $n \ge
2$.  Since $E(K)/nE(K)$ injects into $S^{(n)}(E/K)$, and the latter is
finite and effectively computable, this approach gives an upper bound
for the rank of $E(K)$. However, this upper bound will not be sharp if
the Tate--Shafarevich group $\Sha(E/K)$ contains elements of order
$n$.

Let $p$ be a prime. The Kummer exact sequences for
multiplication-by-$p$ and multiplication-by-$p^2$ on $E$ fit into a
commutative diagram
\[ \xymatrix{ E(K) \ar[r]^{p^2} \ar[d]_p & E(K) \ar@{=}[d] \ar[r] & 
S^{(p^2)}(E/K) \ar[d]^\alpha \ar[r] & \Sha(E/K)[p^2] \ar[r] \ar[d]_p & 0 \\
E(K) \ar[r]^{p} & E(K) \ar[r] & S^{(p)}(E/K) \ar[r] & \Sha(E/K)[p] \ar[r] & 0
\rlap{.} } \]
We therefore have inclusions 
\begin{equation}
\label{incl}
E(K)/p E(K) \subset \Image(\alpha) \subset S^{(p)}(E/K).
\end{equation}
Cassels \cite{CasselsIV} constructed an alternating bilinear pairing 
\begin{equation}
\label{pairing:Sp}
S^{(p)}(E/K) \times S^{(p)}(E/K) \to \Q/\Z 
\end{equation}
whose kernel is the image of $\alpha$.  If we compute this pairing,
and find it is non-trivial, then by~(\ref{incl}) we get a better
upper bound for the rank of $E(K)$ than was obtained by computing
$S^{(p)}(E/K)$. In such cases, we also learn that the $p$-torsion of
$\Sha(E/K)$ is non-trivial.

Cassels~\cite{Cassels98} showed how to compute the
pairing~(\ref{pairing:Sp}) in the case $p=2$. We now generalise to the
case $p=3$.  Therefore, as a starting point for our work, we rely on
the algorithms for computing $S^{(3)}(E/K)$, as described in
\cite{SS}, and for representing its elements as plane cubics, as
described in \cite{descsum}.  Cassels' method (in the case $p=2$)
involves both a local part (computing a certain local pairing), and a
global part (solving conics over the field of definition of a
$2$-torsion point of $E$). Both parts require significant
modification when $p>2$.

In the case $p=2$, the local pairing turns out to be the Hilbert
norm residue symbol. However, since the local pairing is symmetric and
the Hilbert norm residue symbol is skew-symmetric, this cannot be true
for $p > 2$. A further difficulty is that on passing to a finite
extension of local fields, the values of the local pairing are
multiplied by the degree of the field extension. So if
$[K_v(E[p]):K_v]$ is divisible by $p$, then we cannot reduce to the
case, treated in \cite{O`Neil}, where $E$ has all $p$-torsion points
defined over $K_v$. In Section~\ref{sec:loc}, we nonetheless show how
to write the local pairing (for $p$ odd) in terms of Hilbert norm
residue symbols, and make this completely explicit in the case $p=3$.

In Section~\ref{sec:glob}, we generalise the global part of Cassels'
method to the case $p=3$. In fact, we solve a more general problem
about $3 \times 3 \times 3$ cubes (as studied in \cite{BH}, \cite{DG},
\cite{HoThesis}, \cite{Ng}), using the work of Haile \cite{Haile} and
Kuo \cite{Kuo} on the generalisation of Clifford algebras to cubic
forms. Our solution to this more general problem works by reducing it
to that of trivialising a $3 \times 3$ matrix algebra over a field
$L$. In our application to computing the Cassels--Tate pairing, $L$ is
the field of definition of a $3$-torsion point of~$E$.

The problem of trivialising an $n \times n$ matrix algebra (that is,
given structure constants for an $L$-algebra known to be isomorphic to
$\Mat_n(L)$, find such an isomorphism explicitly) is equivalent in the
case $n=2$ to solving a conic. For $n > 2$, this problem has been
studied in \cite[Section 5]{GHPS}, \cite[Paper~III,
Section~6]{descsum}, \cite{IRS}, with the result that practical
algorithms are available if both $n$ and the discriminant of the
number field $L$ are sufficiently small. However, since for us $L$ is
the field of definition of a $3$-torsion point (which typically has
degree $8$), we have so far only been able to compute a few small
examples.

In Section~\ref{sec:ex}, we illustrate our work by computing the
Cassels--Tate pairing on the $3$-Selmer group of a specific elliptic
curve $E/\Q$. 
To make the example interesting $E$ was chosen
from Cremona's tables~\cite{CrTables} so that it does not admit any rational
$3$-isogenies and $\Sha(E/\Q)[3] \not = 0$. To make the computations
practical we also chose $E$ so that the degree $8$ number
field $L$ has reasonably small discriminant. Strictly speaking, we only
compute the pairing up to a global choice of sign, but this does not
matter for applications.

Computing the pairing~(\ref{pairing:Sp}) gives the same information
(in terms of improving our upper bound for the rank) as a
$p^2$-descent.  In \cite{Cassels98}, Cassels claims that his method
(for $p=2$) is more efficient than performing a $4$-descent, as
described in \cite{MSS}. Subject to finding a better algorithm for
trivialising matrix algebras over number fields, our method (for
$p=3$) should also be more efficient than performing a $9$-descent, as
described in \cite{creutz}. One advantage of computing the pairing,
compared to performing a $p^2$-descent directly, is that fewer class
group calculations are required. Another advantage is that we only
need to compute the pairing on a basis for $S^{(p)}(E/K)$, whereas
$p^2$-descent must be run on every element of $S^{(p)}(E/K)$.

The pairing~(\ref{pairing:Sp}) is, in fact, induced by a pairing
\[ \langle~,~\rangle : \Sha(E/K) \times \Sha(E/K) \to \Q/\Z \] and
this is the form in which the Cassels--Tate pairing is usually
written.  Following the terminology in \cite{PS}, the original
definition in \cite[Section 3]{CasselsIV} is called the ``homogeneous
space definition'' (see also \cite[I, Remark 6.11]{ADT}, \cite[Section
2.2]{CTPPS}), whereas the variant used in \cite[Section 6]{CasselsIV}
is called the ``Weil pairing definition'' (see also \cite[I,
Proposition 6.9]{ADT}, \cite[Section 2.2]{CTPPS}).  Both the method in
\cite{Cassels98} and our generalisation use the Weil pairing
definition.

In Section~\ref{sec:CTpair}, we use the description of $H^1(K,E[p])$
in \cite{SS} to make the pairing explicit for $p > 2$.  The formula we
give is for $\langle x,y \rangle$ where $x,y \in \Sha(E/K)$ and
$py=0$. Since we do not require $px = 0$, our work might be described
(following \cite{SD}) as doing a $p^n$-descent for all $n$. We take
$p$ an odd prime, as the case $p=2$ is already described in
\cite{Cassels98}, \cite{Yoga}, \cite{SD}.

The Weil pairing definition was used in \cite{CasselsI}, where Cassels
computed the pairing on the $3$-isogeny Selmer groups of certain
elliptic curves with $j$-invariant $0$. This is currently being
generalised to other isogenies of prime degree by the first author's
student M. van Beek.  The homogeneous space definition has also been
used for explicit computation, most notably in the Magma \cite{magma}
implementation of the pairing on $S^{(2)}(E/\Q)$ due to S. Donnelly.
It might be interesting to investigate how this approach generalises
to the case $p=3$, but we have not done so.

We write $H^i(K,-)$ for the Galois cohomology group $H^i(\Gal(\Kbar/K),-)$,
and $E[p]$ for the kernel of multiplication-by-$p$ on $E(\Kbar)$.
The completion of a number field $K$ at a place $v$ is denoted
$K_v$. We write $M_K$ for the set of all places of $K$.
Since we take $p$ an odd prime, we can ignore the infinite places.

A Magma file containing some of the formulae in Sections~\ref{sec:glob} 
and~\ref{sec:ex} may be found accompanying the arXiv version of this
article. 

\subsection*{Acknowledgements}
We thank Manjul Bhargava, Wei Ho and Hendrik Lenstra for useful
mathematical conversations and for pointing out some of the references.
All computer calculations in Sections~\ref{sec:glob} 
and~\ref{sec:ex} were performed using Magma \cite{magma}.
The second author is grateful for funding from DIAMANT.

\section{The Cassels--Tate pairing}
\label{sec:CTpair}

Let $K$ be a field of characteristic $0$, and $\Kbar$ its algebraic
closure. Let $E/K$ be an elliptic curve and $p$ an odd prime.  The
$p$-torsion subgroup $E[p]$ may be regarded as a 2-dimensional affine
space over $\F_p$. We write $\PP(E[p])$ for the set of lines passing
through $\zeroE$, and $\Lambda$ for the set of lines not passing
through $\zeroE$.  The \'etale algebra of $X$, a finite set with
Galois action, is the $K$-algebra $R = \Map_K(X,\Kbar)$ of all Galois
equivariant maps from $X$ to $\Kbar$. It is a product of field
extensions of $K$, one for each Galois orbit of elements in $X$.  We
also write $\Rbar = R \otimes_K \Kbar = \Map(X,\Kbar)$ for the
$\Kbar$-algebra of all maps from $X$ to $\Kbar$, and let
$\Gal(\Kbar/K)$ act on these maps in the natural way, that is, by
conjugation.

Let $L^+$, $L$, $L'$ and $M$ be the \'etale algebras of $\PP(E[p])$,
$E[p] \setminus \{ \zeroE \}$, $\Lambda$ and
\[ \{ (T,\lin) \in (E[p] \setminus \{ \zeroE \}) \times \Lambda : T
\in \lin \}.\] These are $K$-algebras of dimensions $p+1$, $p^2-1$,
$p^2-1$ and $p(p^2-1)$. There are natural inclusions $L^+ \subset L
\subset M$ and $L' \subset M$. We fix $\nu \in \Z$ a primitive root
mod $p$ and let $\sigma_\nu$ be the generator of $\Aut(L/L^+)$ induced
by multiplication-by-$\nu$ on $E[p]$. The inclusion $L \subset M$
followed by the norm map $N_{M/L'}$ is given by
\[   (T \mapsto \alpha_T)  \mapsto 
       ( \lin \mapsto \textstyle\prod_{T \in \lin} \alpha_T ). \]

Let $w: E[p] \to \mu_p(\Lbar)$ be the map induced by the
Weil pairing. 
This induces a group homomorphism
\begin{equation}
\label{w1}
 w_1 : H^1(K,E[p]) \to L^\mult/(L^\mult)^p. 
\end{equation}
Explicitly, if $\xi \in H^1(K,E[p])$ is represented by a cocycle
$(\sigma \mapsto \xi_\sigma)$ then by Hilbert's theorem 90, there
exists $\gamma \in \Lbar^\mult$ such that $w(\xi_\sigma) =
\sigma(\gamma)/\gamma$ for all $\sigma \in \Gal(\Kbar/K)$.  Then
$\alpha = \gamma^p$ belongs to $L^\mult$ and we define $w_1(\xi) =
\alpha \mod{(L^{\mult})^p}$.

\begin{lemma}
\label{lem:SS}
The map $w_1$ is injective and has image
\[ \left\{ \alpha \in L^\mult/(L^\mult)^p \bigg|
  \begin{array}{r@{\,\,}c@{\,\,}l@{}l}
    \sigma_{\nu}( \alpha ) &\equiv& \alpha^\nu & \mod{(L^{\mult})^p} \\
    N_{M/L'} ( \alpha ) &\equiv& 1 & \mod{(L'^\mult)^p} \end{array} \right\}. \]
\end{lemma}

\begin{proof}
  Injectivity is proved in \cite[Section~3]{DSS} and
  \cite[Corollary~5.1]{SS}. The image is described in
  \cite[Corollary~5.9 and Proposition~5.10]{SS}.
\end{proof}

We now suppose $K$ is a number field. Let $C/K$ be a principal
homogeneous space under $E$. Then $C$ is a smooth curve of genus one
with Jacobian $E$. We further suppose that $C$ is everywhere locally
soluble, that is, $C(K_v) \not= \emptyset$ for all places $v$ of $K$.
We write ``$\summ$'' for the natural isomorphism $\Pic^0(C) \isom E$.
We make frequent use of the fact that a divisor on $C$ is principal if
and only if it has degree $0$ and sum $\zeroE$.

For each $\zeroE \not= T \in E[p]$, there is a degree $0$ divisor
$\aaa_T$ on $C$ with $\summ(\aaa_T) = T$.  Since $C$ is everywhere
locally soluble we can choose the divisors $\aaa_T$ so that the map $T
\mapsto \aaa_T$ is Galois equivariant. The proof of this, as given in
\cite[Lemma~7.1]{CasselsIV} or \cite[Lemma~1]{SD}, uses the
local-to-global principle for the Brauer group of $K(T)$.  Since
$\summ(p \aaa_T) = p T=\zeroE$, there are rational functions $f_T \in
\Kbar(C)$ with $\divv(f_T) = p \aaa_T$.  By Hilbert's theorem 90, we
may scale the $f_T$ so that $f= (T \mapsto f_T)$ is Galois
equivariant.  Then $f$ is an element of $L(C) = L \otimes_K K(C) =
\Map_K(E[p] \setminus \{ \zeroE \},\Kbar(C))$.

The following lemma specifies a scaling of $f$ that is unique up to
multiplication by elements in the image of $w_1$.  We abbreviate
$N_{M(C)/L'(C)}$ as $N_{M/L'}$.

\begin{lemma}
\label{lem:scale-f}
Let $f \in L(C)$ as above. After multiplying $f$ by a suitable element
of $L^\mult$, there exist $r \in L(C)$ and $s \in L'(C)$ such that
\begin{equation}
\label{scale-f}
  \sigma_\nu(f) /f^{\nu} = r^p \qquad  
  \text{ and } \qquad N_{M/L'}(f) = s^p. 
\end{equation}
\end{lemma} 
\begin{proof}
We choose $r \in L(C)$ and $s \in L'(C)$ satisfying
\begin{equation*}
  \divv(r_T) = \aaa_{\nu T} - \nu \aaa_T  \qquad  
  \text{ and } \qquad \divv(s_\lin) = \textstyle\sum_{T \in \lin} \aaa_T.
\end{equation*}
Then~(\ref{scale-f}) holds up to scalars. The construction of $s$ uses
the fact that the points on a line $\lin$ sum to zero, which in turn
depends on the fact $p$ is odd.

To remove the scalars, we use the result of Tate \cite[Lemmas 5.1 and
6.1]{CasselsIV} that, since $C$ is everywhere locally soluble, its
class in $H^1(K,E)$ is divisible by $p$.  If $C$ and $C_1$ correspond
to classes $x$ and $x_1$ in $H^1(K,E)$ with $p x_1 = x$ then there is
a commutative diagram
\[ \xymatrix{ C_1 \ar[r]^\pi \ar[d] & C \ar[d] \\ E \ar[r]^{[p]} & E }
\] where $\pi$ is a morphism defined over $K$, and the vertical maps
are isomorphisms defined over $\Kbar$.  We say that $\pi : C_1 \to C$
is a $p$-covering.  For $\bbb$ a divisor on $E$ we have $\summ([p]^*
\bbb) = p \summ(\bbb)$.  So there exists $g \in L(C_1)$ with
$\divv(g_T) = \pi^* \aaa_T$.  We now scale $f$ so that $\pi^*f = g^p$,
and scale $r$ and $s$ so that
\begin{equation*}
  \pi^* r = \sigma_\nu(g) /g^{\nu}  \qquad  
  \text{ and } \qquad \pi^* s = N_{M/L'} (g).
\end{equation*}
It is then easy to check that (\ref{scale-f}) holds exactly.
\end{proof}

Let $v$ be a place of $K$. By the Weil pairing, cup product and 
the local invariant map there is a pairing 
\begin{equation}
\label{tatepairing}
(~,~)_v : H^1(K_v,E[p]) \times H^1(K_v,E[p]) \to \Q/\Z.
\end{equation}
It is known (see \cite[I, Theorem 3.2]{ADT},\cite{T})
that $(~,~)_v$ is symmetric and 
non-degenerate, and that the image of $E(K_v)/p E(K_v)$ is a 
maximal isotropic subspace. The last of these facts is referred to 
as {\em Tate local duality}. The local analogue of~(\ref{w1}) is
a map $w_{1,v}$ that fits in a commutative diagram
\[ \xymatrix{ 
 H^1(K,E[p]) \ar[r]^-{w_{1}}  \ar[d]_{\Res_v} & 
L^\mult/(L^\mult)^p \ar[d] \\
 H^1(K_v,E[p]) \ar[r]^-{w_{1,v}} & L_v^\mult/(L_v^\mult)^p } \]
where $L_v = L \otimes_K K_v$. We write $[~,~]_v$ for the pairing
induced by $(~,~)_v$ on the image of $w_{1,v}$. Since the 
local invariants of an element in $\Br(K)$ sum to zero, we have
the ``product formula'' \footnote{This is analogous to
the product formula for the Hilbert norm residue symbol.}
\begin{equation}
\label{prodform}
\sum_{v \in M_K} [ \alpha, \beta]_v = 0 
\end{equation}
for all $\alpha,\beta \in \Image(w_1)$.

\begin{theorem}
\label{thm:ctp}
Let $x,y \in \Sha(E/K)$ with $py = 0$. Let $C/K$ be a principal
homogeneous space under $E$ representing $x$, and let $\eta \in
S^{(p)}(E/K)$ be an element that maps to $y$.  Let $f \in L(C)$ be
scaled as in Lemma~\ref{lem:scale-f}, and for each place $v$ of $K$
choose a point $P_v \in C(K_v)$, avoiding the zeros and poles of the
rational functions $f_T$.  Then the Cassels--Tate pairing is given by
\begin{equation}
\label{def:ctp}
  \langle x, y \rangle = \sum_{v \in M_K} [ f(P_v), w_{1} (\eta) ]_v. 
\end{equation}
\end{theorem}
\begin{proof}
  We start by checking that~(\ref{def:ctp}) is well-defined as a
  function of $x$ and $\eta$. Lemmas~\ref{lem:SS}
  and~\ref{lem:scale-f} show that $f(P_v)$ is in the image of
  $w_{1,v}$, and so is a valid argument for $[~,~]_v$.  It is shown in
  \cite[Theorem 2.3]{Schaefer} that evaluating $f$ on degree $0$
  divisors gives an explicit realisation of the connecting map
  $\delta_v : E(K_v)/p E(K_v) \to H^1(K_v,E[p])$. So by Tate local
  duality each of the summands in~(\ref{def:ctp}) is independent of
  the choice of $P_v \in C(K_v)$.  The pairing~(\ref{def:ctp}) is also
  independent of the choice of scaling of $f$ as in
  Lemma~\ref{lem:scale-f}, by the product formula~(\ref{prodform}).

  Next, we check that~(\ref{def:ctp}) agrees with one of the standard
  definitions of the Cassels--Tate pairing.  By the proof of
  Lemma~\ref{lem:scale-f}, there is a $p$-covering $\pi : C_1 \to C$
  defined over $K$, and we may scale $f$ so that $\pi^* f = g^p$ for
  some $g \in L(C_1)$.  Since $C$ is everywhere locally soluble, for
  each place $v$ of $K$ there is a $p$-covering $\pi_v : C_{1,v} \to
  C$ defined over $K_v$ with $C_{1,v}(K_v) \not= \emptyset$.  Now
  $\pi_v : C_{1,v} \to C$ is the twist of $\pi : C_1 \to C$ by some
  $\xi_v \in H^1(K_v,E[p])$.  The ``Weil pairing definition'' of the
  Cassels--Tate pairing (see \cite[Section 6]{CasselsIV},
  \cite[Section 2.2]{CTPPS} or \cite[I, Proposition 6.9]{ADT}) says
  that
\begin{equation}
\label{def:ctp-wp}
  \langle x, y \rangle = \sum_{v \in M_K} ( \xi_v , \Res_v \eta )_v. 
\end{equation}

Let $P_v \in \pi_v( C_{1,v}(K_v))$.  Since $C_{1,v}(K_v)$ is infinite
we may assume that $P_v$ is not a zero or pole of $f$.  By
Lemma~\ref{lem:covers} applied over $K_v$ we have
\begin{equation*}
w_{1,v} (\xi_v) \equiv f(P_v) \mod{(L_v^\mult)^p}.
\end{equation*}
It follows that the pairings~(\ref{def:ctp}) 
and~(\ref{def:ctp-wp}) are the same.
\end{proof}

\begin{lemma}
\label{lem:covers}
Let $C/K$ and $f \in L(C)^\mult$ be as before.
Let $\pi_1 : C_1 \to C$ be a $p$-covering, and 
$\pi_2 : C_2 \to C$ its twist by $\xi \in H^1(K,E[p])$. 
Suppose that for $i=1,2$ we have $\pi_i^*f = \alpha_i g_i^p$ for some 
$\alpha_i \in L^\mult$ and $g_i \in L(C_i)$. Then 
$w_1(\xi) \equiv \alpha_2/\alpha_1 \mod{(L^\mult)^p}$. 
\end{lemma}

\begin{proof}
The proof is closely related to that of \cite[Theorem 2.3]{Schaefer}.
There is an isomorphism $\psi$ defined over 
$\Kbar$ making the following diagram commute.
\[ \xymatrix{ 
 C_2 \ar[r]^{\pi_2}  \ar[d]_{\psi} & C \ar@{=}[d] \\
 C_1 \ar[r]^{\pi_1} & C } \]
Then $\xi\in H^1(K,E[p])$ is represented by a cocycle 
$(\sigma \mapsto \xi_{\sigma})$ where $\sigma(\psi)\psi^{-1}$
is translation by $\xi_\sigma \in E[p]$. 
We have $\psi^* g_1 = \gamma g_2$ for some $\gamma \in \Lbar^\mult$. 
By definition of the Weil pairing 
\[ w(\xi_\sigma) g_1 = (\sigma(\psi) \psi^{-1})^* g_1 = 
(\sigma(\gamma)/\gamma) g_1. \]
It follows that $w(\xi_\sigma) = \sigma(\gamma)/\gamma$ and so
 $w_1(\xi) \equiv \gamma^p \equiv \alpha_2/\alpha_1 \mod{(L^\mult)^p}$.
\end{proof}

The formula~(\ref{def:ctp}) is in fact a finite sum. This may be seen
by Tate local duality, and the following lemma.  We write $\OO_v
\subset L_v$ for the product of valuation rings of the constituent
fields, and $l_v$ for the product of residue fields.

\begin{lemma}
\label{lem:finite}
Let $C/K$ and $f \in L(C)^\mult$ be as in Theorem~\ref{thm:ctp}.
\begin{enumerate}
\item If $v \nmid p \infty$ is a prime of good reduction for $E$ then
\[ \Image(w_{1,v} \circ \delta_v) = \Image (w_{1,v}) \cap \OO_v^\mult/(\OO_v^\mult)^p. \]
\item If $v \nmid p \infty$ is a prime of good reduction for $C$, and
  $f$ reduces mod $v$ to $\widetilde{f} \in l_v(\widetilde{C})^\mult$
  then
  \[ f(P_v) \in \Image (w_{1,v}) \cap \OO_v^\mult/(\OO_v^\mult)^p \]
  for all $P_v \in C(K_v)$ avoiding the zeros and poles of the $f_T$.
\end{enumerate}
\end{lemma}
\begin{proof}
  (i) It is well known that $\Image(\delta_v)$ is the unramified
  subgroup of $H^1(K_v,E[p])$. See for example \cite[Proposition
  3.2]{SS}, where this is proved under slightly weaker assumptions on
  $v$.  We then use that $\OO_v^\mult/(\OO_v^\mult)^p$ is the kernel
  of the natural map
  \[ (L \otimes_K K_v)^\mult / \{ \text{$p$th powers} \} \to (L
  \otimes_K K^\nr_v)^\mult / \{ \text{$p$th powers} \}.  \] (ii) By
  (i) and the proof of Theorem~\ref{thm:ctp} it suffices to prove this
  for just one choice of $P_v$. If the residue field $k_v$ of $K_v$ is
  sufficiently large then there exists $\widetilde{P}_v \in
  \widetilde{C}(k_v)$ avoiding the zeros and poles of the
  $\widetilde{f}_T$.  We then use Hensel's lemma to lift
  $\widetilde{P}_v$ to $P_v \in C(K_v)$, and see that $f(P_v)$ is a
  unit. If $k_v$ is too small then we rectify this by making an
  unramified extension.
\end{proof}

We would like to use Theorem~\ref{thm:ctp} to compute the
Cassels--Tate pairing. There are essentially two problems.
\begin{itemize}
\item Computing the local pairing $[~,~]_v$. This is the
subject of Section~\ref{sec:loc}.
\item Computing the rational functions $f_T$. In Section~\ref{sec:glob}
we describe a method for doing this in the case where $p=3$
and $C$ is a plane cubic.
\end{itemize}

In the case $p=2$, the pairing $[~,~]_v$ can be written as a product of
Hilbert norm residue symbols. This is used implicitly in Cassels'
paper \cite{Cassels98}, and a detailed proof is given in \cite{Yoga}.
Our generalisation to the case $p=3$ is necessarily more complicated
since $[~,~]_v$ is symmetric, whereas the Hilbert norm residue symbol
is skew-symmetric.

Cassels' method for computing the $f_T$ requires us to solve conics
over the field of definition of a $2$-torsion point on $E$. The conics
arise by a geometric construction that seems very special to the case
$p=2$. Nonetheless, we have found a practical method for reducing the
problem in the case $p=3$ to that of ``trivialising a matrix algebra''
over the field of definition of a $3$-torsion point on $E$.

\section{Computing the local pairing}
\label{sec:loc}

We keep the notation of Section~\ref{sec:CTpair}, up to and including
Lemma~\ref{lem:SS}, but now take $K$ a $p$-adic field. In this section,
we compute the local pairing $[~,~]_K$ on the image of $w_1$. Since
$[~,~]_K$ is symmetric, and $p$ is odd, it is equivalent to compute
the quadratic form $\varphi_K : \Image(w_1) \to \recip{p}\Z/\Z$
satisfying
\begin{equation}
\label{sym-quad}
 [\alpha,\beta]_K = \varphi_K(\alpha\beta) - 
     \varphi_K(\alpha) - \varphi_K(\beta) 
\end{equation}
for all $\alpha,\beta \in \Image(w_1)$.

We fix $\zeta_p \in \Kbar$ a primitive $p$th root of unity.

Let $T_1, \ldots, T_m \in E[p] \setminus \{ \zeroE \}$ be
representatives for the $\Gal(\Kbar/K)$-orbits.  Then $L = L_1 \times
\ldots \times L_m$ where $L_j = K(T_j) \subset \Kbar$.  We write
$\{~,~\}_j$ for the Hilbert norm residue symbol on
$L_j(\zeta_p)^\mult/(L_j(\zeta_p)^{\mult})^{p}$. This takes values in
$\mu_p$.

Let $\iota : L' \otimes_K K(\zeta_p) \isom L \otimes_K K(\zeta_p)$ be
the isomorphism induced by the bijection
\begin{equation}
\label{lambda:bij}
\begin{aligned}
E[p] \setminus \{ \zeroE \} & \leftrightarrow \Lambda \\
T & \mapsto \{ S \in E[p] : e_p(S,T) = \zeta_p \}
\end{aligned}
\end{equation}
This depends on the choice of $\zeta_p$.

Let $\Ind_{\zeta_p} : \mu_p \isom \recip{p}\Z/\Z$ be the isomorphism
that maps $\zeta_p \mapsto \recip{p}$.

\begin{theorem}
\label{localpairing}
Assume $p$ is an odd prime. Let $\alpha \in L$ represent an element in
the image of $w_1 : H^1(K,E[p]) \to L^\mult/(L^\mult)^p$.  Then we may
associate to $\alpha$ an element $\alpha' \in L'$ such that for each
$1 \le j \le m$,
\[ [L_j(\zeta_p):K] \varphi_K(\alpha) = \left\{ \begin{array}{ll}
    \Ind_{\zeta_p} \{\alpha(T_j), \iota (\alpha')(T_j) \}_j &
    \text{ if } \iota (\alpha')(T_j) \not=0, \\
    0 & \text{ otherwise.} \end{array} \right. \] If $p=3$ then we may
take
\[ \alpha' = \Tr_{M/L'}(\alpha) - 3 N_{M/L'}(\alpha)^{1/3} \] where
the cube root is chosen as specified in Proposition~\ref{cuberoot}(i).
\end{theorem}

\begin{Remark}
Theorem~\ref{localpairing} can be used to compute $\varphi_K(\alpha)$ 
in all cases, since the degree 
$[L_j(\zeta_p):K]$ is coprime to $p$ for at least one $j$.
\end{Remark}

The proof of Theorem~\ref{localpairing} uses several constructions
from \cite{descsum}, the most important of which is described in
Proposition~\ref{csa} below.

Following \cite{descsum}, let $R=\Map_K(E[p],\Kbar)$ be the \'{e}tale
algebra of $E[p]$, and let $\overline{R}=\Map(E[p],\Kbar)=R\otimes_K
\Kbar$.  Writing $E[p]=\{\zeroE\}\cup (E[p]\setminus\{\zeroE\})$ there
are decompositions $R=K\times L$ and $\Rbar=\Kbar\times \Lbar$.  The
Weil pairing induces a map $w:E[p]\rightarrow \mu_p(\Rbar)$ given by
$S\mapsto (T\mapsto e_p(S,T))$.  There is then an exact sequence
\[  0 \ra E[p] \stackrel{w}{\ra} \Rbar^\mult \stackrel{\partial}{\ra}
            \partial \Rbar^\mult  \ra  0 \]
where the map 
\[\partial:\overline{R}^{\mult}\rightarrow (\overline{R}\otimes_{\Kbar} \overline{R})^{\mult}=\Map(E[p]\times E[p],\Kbar^{\mult})\]
is defined by $\partial\beta (S,T)=\beta(S)\beta(T)/\beta(S+T)$ for
all $S,T\in E[p]$.  Taking Galois cohomology gives an injective group
homomorphism
\[w_2:H^1(K,E[p])\rightarrow (R\otimes_K R)^{\mult}/\partial R^{\mult}.\] 

Let $\gamma\in \Lbar^\mult$ be as described in the definition of $w_1$
(see Section~\ref{sec:CTpair}). We extend $\gamma$ to an element of
$\Rbar^\mult$ by setting $\gamma(\zeroE)=1$. Then
$\rho=\partial\gamma\in (R\otimes_K R)^{\mult}$ and
$w_2(\xi)=\rho\mod{\partial R^{\mult}}$.  It is convenient to
summarise this situation as follows.
\begin{definition}
\label{compatible}
Let $\xi\in H^1(K,E[p])$. We call $\alpha\in L^{\mult}$, $\rho\in
(R\otimes_K R)^{\mult}$ {\em compatible} representatives for
$w_1(\xi)$ and $w_2(\xi)$ if there exist a cocycle $(\sigma\mapsto
\xi_{\sigma})\in Z^1(K,E[p])$ representing $\xi$ and $\gamma\in
\Rbar^{\mult}$ such that all of the following conditions are
satisfied.
\begin{enumerate}
\item For all $\sigma\in G_K$ and all $T\in E[p]$, we have
  $e_p(\xi_{\sigma},T)=(\sigma\gamma/\gamma)(T)$;
\item $\gamma(\zeroE)=1$ and for all $T\in E[p]\setminus\{\zeroE\}$,
  $\gamma(T)^p=\alpha(T)$;
\item $\rho=\partial\gamma$.
\end{enumerate}
\end{definition}

\begin{definition}
  Let $\rho$ represent an element in the image of $w_2$. Following
  \cite{descsum}, we define a new multiplication $*_{\rho}$ on the
  $\Kbar$-vector space $\Rbar$ as follows. For all $f,g\in \Rbar$ and
  for all $T\in E[p]$,
\begin{equation*}
  (f*_{\rho}g)(T)=\sum_{T_1+T_2=T}{e_p^{1/2}(T_1,T_2)\rho(T_1,T_2)f(T_1)g(T_2)}
\end{equation*}
where $e_p^{1/2}(T_1,T_2)\in \mu_p$ is the square root of the Weil
pairing. For each $T\in E[p]$, let $\delta_T\in\overline{R}$ be the
indicator function \[\delta_T(S)=\left\{\begin{array}{ll}1 & \textrm{
      if }\ S=T,\\ 0 & \textrm{ if }\ S\neq T.\end{array}\right.\] The
indicator function $\delta_\zeroE$ is the identity for the
multiplication $*_{\rho}$. For all $S,T\in E[p]$, we have
\[\delta_S*_{\rho}\delta_T=e_p^{1/2}(S,T)\rho(S,T)\delta_{S+T}\] and
therefore
\begin{equation}
  \label{conjugation}
  \delta_S*_{\rho}\delta_T=e_p(S,T)\delta_T*_{\rho}\delta_S
\end{equation}
since $\rho$ is symmetric and $e_p$ is skew-symmetric.
\end{definition}

\begin{lemma}
\label{powers}
Let $\xi\in H^1(K,E[p])$ and let $\alpha\in L^{\mult}$, $\rho\in
(R\otimes_K R)^{\mult}$ be compatible representatives for $w_1(\xi)$
and $w_2(\xi)$ respectively. Then for all $T\in E[p]\setminus
\{\zeroE\}$, we have
\[\delta_T^p=\underbrace{\delta_{T}*_{\rho}\delta_T*_{\rho}\dots *_{\rho}\delta_T}_{p \textrm{ times}} =\alpha(T)\delta_{\zeroE}.\]
Therefore, $\delta_T$ is invertible with respect 
to the multiplication $*_{\rho}$.
\end{lemma}

\begin{proof}
  Let $\gamma\in\Rbar^{\mult}$ be as in Definition~\ref{compatible}.
  Then
\begin{eqnarray*}
  \underbrace{\delta_{T}*_{\rho}\delta_T*_{\rho}\dots *_{\rho}\delta_T}_{p \textrm{ times}} &=&\prod_{i=1}^{p-1}{\rho(T,iT)}\delta_{\zeroE}
  =\prod_{i=1}^{p-1}{\frac{\gamma(T)\gamma(iT)}{\gamma((i+1)T)}\delta_\zeroE}\\
  &=&\frac{\gamma(T)^p}{\gamma(pT)}\delta_\zeroE=\frac{\alpha(T)}{\gamma(\zeroE)}\delta_\zeroE=\alpha(T)\delta_\zeroE.
\end{eqnarray*}
Since $\alpha\in L^{\mult}$, we have $\alpha(T)\in \Kbar^{\mult}$ and
therefore $\delta_T$ is invertible.
\end{proof}

\begin{definition}
  Denote by $A_{\rho}$ the algebra which is the $K$-vector space $R =
  \Map_K(E[p],\Kbar)$ equipped with the new multiplication $*_{\rho}$.
\end{definition}

\begin{proposition}
\label{csa}
Let $\xi\in H^1(K,E[p])$ and choose $\rho\in(R\otimes_K R)^{\mult}$
such that $w_2(\xi) = \rho \mod{\partial R^{\mult}}$. Then $A_{\rho}$
is a central simple algebra of dimension $p^2$ over $K$ and
$\varphi_K(w_1(\xi))=\inv_K(A_{\rho})$.
\end{proposition}

\begin{proof}
  Let $\Ob_K : H^1(K,E[p]) \to \Br(K)$ be the period-index obstruction
  map, as defined in \cite{descsum}, \cite{O`Neil}. In \cite{Z}, it is
  shown that the pairing
  \[ H^1(K,E[p]) \times H^1(K,E[p]) \to \Br(K) \] defined by
  cup-product and the Weil pairing, is also given by
  \[ (\xi,\eta) \mapsto \Ob_K(\xi+\eta) - \Ob_K(\xi) - \Ob_K(\eta). \]
  Comparing with~(\ref{sym-quad}), we have
  $\varphi_K(w_1(\xi))=\inv_K(\Ob_K(\xi))$.

  In \cite[Section 4.3]{descsum}, it is shown that $\Ob_K(\xi)$ is
  represented by a certain central simple algebra of dimension $p^2$
  over $K$.  Then Lemmas 4.5, 3.10 and 3.11 of \cite{descsum} show
  that this algebra is $A_\rho$.
\end{proof}

We will make use of the following lemma in our study of the central simple algebra $A_{\rho}$. 

\begin{lemma}
\label{scalar matrix}
Suppose that $A,B\in \GL_n(K)$ satisfy $A^n=B^n=I_n$ and $AB=\zeta_n
BA$ for some primitive $n$th root of unity $\zeta_n\in K$.
\begin{enumerate}
\item If $C\in \Mat_n(K)$ satisfies $CB=\zeta_n BC$, then $C^n$ is a
  scalar matrix.
\item The matrices $A^rB^s$ for $0\leq r,s\leq n-1$ form a basis for
  $\Mat_n(K)$ as a $K$-vector space.
\end{enumerate}
\end{lemma} 

\begin{proof}
  (i) 
  Since $B^n=I_n$,
  all eigenvalues of $B$ are $n$th roots of unity. If $v$ is an
  eigenvector with eigenvalue $\lambda$, then $A^{-m}v$ is an
  eigenvector with eigenvalue $\zeta_n^m\lambda$. So $B$ has $n$
  distinct eigenvalues, namely the $n$th roots of unity. Changing
  basis, we may assume that $B=(b_{ij})$ is a diagonal matrix with
  $b_{ii}=\zeta_n^i$. If $C\in \Mat_n(K)$ satisfies $CB=\zeta_n BC$,
  then $C$ is of the form $(c_{ij})$ where $c_{ij}=0$ unless $j\equiv
  i+1\pmod{n}$. Therefore, $C^n=c_{12}c_{23}\dots c_{(n-1)
    n}c_{n1}I_n.$

  (ii) Both $A$ and $B$ act by conjugation on $\Mat_n(K)$.  The matrix
  $A^rB^s$ is an eigenvector with eigenvalue $\zeta^s$ for conjugation
  by $A$, and also an eigenvector with eigenvalue $\zeta^{-r}$ for
  conjugation by $B$. Eigenvectors with distinct eigenvalues are
  linearly independent. Thus, the matrices $A^rB^s$ for $0\leq r,s\leq
  n-1$ are a $K$-basis for $\Mat_n(K)$.
\end{proof}

\begin{proposition}
\label{prop:getscalar}
Let $\rho\in(R\otimes_K R)^{\mult}$ represent an element in the image
of $w_2$. For each $\lin\in \Lambda$, let
$\delta_{\lin}=\sum_{S\in\lin}{\delta_S}$ be the indicator function of
$\lin$. Then \[\delta_{\lin}^p=\underbrace{\delta_{\lin}*_{\rho}\cdots
  *_{\rho}\delta_{\lin}}_{p\ \textrm{times}}\in\Kbar\delta_\zeroE.\]
\end{proposition}

\begin{proof}
  We have $\delta_{\lin}\in A_{\rho}\otimes \Kbar$, which is the
  $\Kbar$-vector space $\Map(E[p],\Kbar)$ equipped with the
  multiplication $*_{\rho}$. Proposition \ref{csa} tells us that
  $A_{\rho}$ is a central simple algebra of dimension $p^2$ over $K$,
  so $A_{\rho}\otimes \Kbar\cong \Mat_p(\Kbar)$. Under this
  isomorphism, the element $\delta_\zeroE$ is identified with the
  identity matrix $I_p$.  By~(\ref{lambda:bij}), there exists $T\in
  E[p] \setminus \{ \zeroE \}$ such that $\lin=\{ S \in E[p] :
  e_p(S,T) = \zeta_p \}$. Let $S\in\lin$. Then equation
  \eqref{conjugation} gives
\[\delta_S*_{\rho}\delta_T=e_p(S,T)\delta_T*_{\rho}\delta_S=\zeta_p\delta_T*_{\rho}\delta_S,\]
and therefore
$\delta_\lin*_{\rho}\delta_T=\zeta_p\delta_T*_{\rho}\delta_\lin$.  By
Lemma~\ref{powers}, $\delta_S$ and $\delta_T$ are invertible. Now apply
the first part of Lemma~\ref{scalar matrix}, with $A$ and $B$ scalar
multiples of $\delta_S$ and $\delta_T$, and $C=\delta_{\lin}$ to see
that $\delta_{\lin}^p\in\Kbar\delta_\zeroE$.
\end{proof}

By Proposition~\ref{prop:getscalar}, there exists $\alpha'\in L'$
defined by \[\delta_{\lin}^p=\delta_{\lin}*_{\rho}\cdots
*_{\rho}\delta_{\lin} =\alpha'(\lin)\delta_\zeroE.\] We prove
Theorem~\ref{localpairing} for this choice of $\alpha'$.  The key step
is the calculation of the local invariant of the central simple
algebra $A_{\rho}$. This is achieved by writing $A_{\rho}$ as a cyclic
algebra (after a field extension). To this end, we recall the
definition of a cyclic algebra and its relation to the
Hilbert norm residue symbol, as defined in \cite{Serre}.

\begin{definition}
  Let $\ell/k$ be a cyclic field extension of degree $n$. Let $\sigma$
  be a generator of $\Gal(\ell/k)$ and let $b\in k^{\mult}$. Let
  $\chi:G_k\twoheadrightarrow\Gal(\ell/k)\rightarrow
  \recip{n}\bbZ/\bbZ$ be the continuous character of the absolute
  Galois group of $k$ which factors through $\Gal(\ell/k)$ and sends
  $\sigma$ to $\recip{n}\pmod{\bbZ}$. Then the \emph{cyclic algebra}
  $(\chi,b)$ is defined as
  $$(\chi,b)=\left\{\sum_{i=0}^{n-1}{a_iv^i}\mid a_i\in\ell\right\}$$
  with multiplication $v^n=b$ and $vav^{-1}=\sigma(a)$ for all
  $a\in\ell$. The algebra $(\chi,b)$ is a central simple algebra over
  $k$ of dimension $n^2$.
\end{definition}
The following definition of the Hilbert norm residue symbol is given
in \cite[Ch.~XIV, \S2]{Serre}.
\begin{definition}
\label{Hilbertsymbol}
Suppose that $K$ contains a primitive $n$th root of unity $\zeta_n$.
Let $a,b\in K^{\mult}$ and let $\alpha\in \Kbar$ satisfy $\alpha^n=a$.
Define a continuous character $\chi_a:G_K\rightarrow
\recip{n}\bbZ/\bbZ$ by $\chi_a: (\alpha\mapsto \zeta_n^i\alpha)\mapsto
i/n\pmod{\bbZ}$.  Then the {\em Hilbert norm residue symbol}
$\{a,b\}_K$ is defined as $\{a,b\}_K=\zeta_n^{n\inv_K(\chi_a,b)}$.
\end{definition}
Thus, if we can express the central simple algebra $A_{\rho}$ as a
cyclic algebra then, by Proposition \ref{csa}, we will have reduced
the problem of computing $\varphi_K$ to a Hilbert symbol computation.
There are well-known explicit formulae for the Hilbert norm residue
symbol for extensions of prime degree. See, for example, \cite{Serre}.

\begin{lemma}
\label{field extn}
Let $F/K$ be a finite extension of fields. Let $\alpha\in L^{\mult}$
represent an element in the image of $w_1$.  Then
$\varphi_F(\alpha)=[F:K]\varphi_K(\alpha).$
\end{lemma}

\begin{proof}
  Write $L_F=L\otimes_K F=\Map_{F}(E[p]\backslash\{\zeroE\},\Kbar)$.
  The natural inclusion $L\hookrightarrow L_F$ gives rise to a natural
  map $L^{\mult}/(L^{\mult})^p\rightarrow L_F^{\mult}/(L_F^{\mult})^p$
  which makes the following diagram commute.
\begin{equation}
\label{restrictF}
\begin{aligned}
  \xymatrix{L^{\mult}/(L^{\mult})^p\ar[d]&
    H^1(K,E[p])\ar[l]_{w_1}\ar[r]^-{\Ob_K}\ar[d]_{\Res} & \Br(K)
    \ar[r]^-{\inv_K}\ar[d]_{\Res}&
    \bbQ/\bbZ\ar[d]^{\times[F:K]}\\
    L_F^{\mult}/(L_F^{\mult})^p &
    H^1(F,E[p])\ar[l]_{w_1}\ar[r]^-{\Ob_F}& \Br(F) \ar[r]^-{\inv_F}&
    \bbQ/\bbZ }
\end{aligned}
\end{equation}
So if $\xi\in H^1(K,E[p])$, and $\alpha\in L^{\mult}$ represents
$w_1(\xi)$, then the same $\alpha$ also represents $w_1(\Res \, \xi)$.
Thus,
\[\varphi_F(\alpha)=\inv_F(\Ob_F(\Res \, \xi))=[F:K]\inv_K(\Ob_K(\xi))
=[F:K]\varphi_K(\alpha).\]
\end{proof}

Lemma~\ref{field extn} shows that, in proving the first part of
Theorem~\ref{localpairing}, we are free to replace $K$ by
$L_j(\zeta_p)$.  So it suffices to prove the following special case.

\begin{theorem}
\label{mainthm2}
Suppose that $T\in E[p]\setminus\{\zeroE\}$ is defined over $K$ and
that $\zeta_p\in K$. Let $\xi\in H^1(K,E[p])$ and let $\alpha\in
L^{\mult}$, $\rho\in (R\otimes_K R)^{\mult}$ be compatible
representatives for $w_1(\xi)$ and $w_2(\xi)$ respectively. Write
$\{~,~\}_K$ for the Hilbert norm residue symbol on \vspace{-3ex}
$K^{\mult}/(K^{\mult})^p$ taking values in $\mu_p$.  Define
$\alpha'\in L'$ by
$\delta_{\lin}^p=\underbrace{\delta_{\lin}*_{\rho}\cdots
  *_{\rho}\delta_{\lin}}_{p\
  \textrm{times}}=\alpha'(\lin)\delta_\zeroE$ for all
$\lin\in\Lambda$. Then
\[ \varphi_K(\alpha) = 
\left\{ \begin{array}{ll} \Ind_{\zeta_p} \{\alpha(T),
\iota (\alpha')(T) \}_K & 
\text{ if } \iota (\alpha')(T) \not=0, \\
0 & \text{ otherwise.} \end{array} \right. \]
\end{theorem}

\begin{proof}
  Proposition \ref{csa} states that $A_{\rho}$ is a central simple
  algebra of dimension $p^2$ over $K$, and that
  $\varphi_K(\alpha)=\inv_K(A_{\rho})$. The Artin-Wedderburn Theorem
  tells us that either $A_{\rho}$ is a division ring or $A_{\rho}\cong
  \Mat_p(K)$ has local invariant zero.  The multiplication on
  $A_{\rho}$ is understood to be that given by $*_{\rho}$ and
  henceforth we omit $*_{\rho}$ from the notation.  Let
  $\linspec=\{S\in E[p] : e_p(S,T) = \zeta_p \}\in\Lambda$ and recall
  that $\delta_{\linspec}=\sum_{S\in\linspec}{\delta_S}$ is the
  indicator function of $\linspec$. By definition of $\iota$, we have
  $\iota (\alpha')(T)=\alpha'(\linspec)$. The element $\alpha'\in L'$
  is defined by $\delta_{\lin}^p=\alpha'(\lin)\delta_\zeroE$ for all
  $\lin\in\Lambda$. So $\delta_{\linspec}$ is invertible if and only
  if $\iota (\alpha')(T)=\alpha'(\linspec)\neq 0$. If
  $\delta_{\linspec}$ is not invertible, then $A_{\rho}$ is not a
  division ring and therefore $\varphi_K(\alpha)=\inv_K(A_{\rho})=0$.
  From now on, we will assume that $\delta_{\linspec}$ is invertible.
  Applying \eqref{conjugation} to each $S\in\linspec$, we see that
  $\delta_{\linspec}\delta_T=\zeta_p\delta_T\delta_{\linspec}$.
  Lemma~\ref{powers} shows that $\delta_T^p=\alpha(T)\delta_\zeroE$
  and consequently $\delta_T$ is invertible. The second part of
  Lemma~\ref{scalar matrix} implies that the elements
  $\delta_{\linspec}^r\delta_T^s$ for $0\leq r,s\leq p-1$ are linearly
  independent over $\Kbar$ and therefore form a $K$-basis for~$A_{\rho}$.

  First, suppose that $\alpha(T)\notin (K^{\mult})^p$. In this case,
  $\delta_T$ generates a degree $p$ cyclic extension isomorphic to
  $K(\sqrt[p]{\alpha(T)})/K$ inside $A_{\rho}$. Define
  $\chi=\chi_{\alpha(T)}:G_K\rightarrow \recip{p}\bbZ/\bbZ$ as in
  Definition~\ref{Hilbertsymbol} and observe that $A_{\rho}\cong
  (\chi,\alpha'(\linspec))=(\chi,\iota (\alpha')(T))$ is a cyclic
  algebra with local invariant equal to
  $\Ind_{\zeta_p}\{\alpha(T),\iota (\alpha')(T)\}_K$.

  Now suppose, on the contrary, that $\alpha(T)\in (K^{\mult})^p$. In
  this case, the Hilbert symbol $\{\alpha(T),\iota (\alpha')(T)\}_K$
  is trivial and also $\delta_T-\sqrt[p]{\alpha(T)}\delta_\zeroE$ is a
  zero divisor in $A_{\rho}$, whereby
  $\varphi_K(\alpha)=\inv_K(A_{\rho})=0$.
\end{proof}

In order to complete the proof of Theorem \ref{localpairing}, it
remains to characterise $\alpha'$ in the special case $p=3$.
\begin{proposition}
\label{p=3}
Let $\xi\in H^1(K,E[3])$ and let $\alpha\in L^{\mult}$, $\rho\in
(R\otimes_K R)^{\mult}$ be compatible representatives for $w_1(\xi)$
and $w_2(\xi)$ respectively. If $\lin =\{S_1,S_2,S_3\} \in \Lambda$
then $\delta_{\lin}^3 = \alpha'(\lambda) \delta_\zeroE$ where
\[ \alpha'(\lin) = \alpha(S_1) + \alpha(S_2) + \alpha(S_3) - 3
\rho(S_1,S_2) \rho(S_3,-S_3). \]
\end{proposition}

\begin{proof}
  Since $\delta_{\lin}^3\in \Kbar\delta_\zeroE$, only the terms
  $\delta_{S_i}\delta_{S_j}\delta_{S_k}$ where $S_i+S_j+S_k=\zeroE$
  make a contribution.  Since the points on a line sum to zero, we
  have $S_1+S_2+S_3=\zeroE$ and therefore,
\begin{eqnarray*}
  \delta_{\lin}^3&=&\big(\sum_{S\in\lin}{\delta_S}\big)^3= \sum_{S\in\lin}{\delta_S^3}+\sum_{\{i,j,k\}=\{1,2,3\}}{\delta_{S_i}\delta_{S_j}\delta_{S_k}}\\
&=&\sum_{S\in\lin}{\delta_S^3}+\sum_{i<j}{(\delta_{S_i}\delta_{S_j}+\delta_{S_j}\delta_{S_i})\delta_{-S_i-S_j}}\\
&=&\sum_{S\in\lin}{\delta_S^3}+\sum_{i<j}\left(e_3^{1/2}(S_i,S_j)+e_3^{1/2}(S_j,S_i)\right)\rho(S_i,S_j)\delta_{S_i+S_j}\delta_{-S_i-S_j}
\end{eqnarray*}
Since $\lin$ does not pass through zero, $e_3^{1/2}(S_i,S_j)$ is a
primitive cube root of unity and, consequently,
$e_3^{1/2}(S_i,S_j)+e_3^{1/2}(S_j,S_i)=-1$. Lemma \ref{powers} shows
that $\delta_S^3=\alpha(S)\delta_\zeroE$ for all $S\in\lin$. Therefore,
$\delta_{\lin}^3 = \alpha'(\lambda) \delta_\zeroE$ where
\[ \alpha'(\lin) = \sum_{S\in\lin}{\alpha(S)} -
\sum_{i<j}{\rho(S_i,S_j) \rho(S_i+S_j,-S_i-S_j)}.\]

Let $\gamma\in\Rbar^{\mult}$ be as in Definition \ref{compatible}.
For $i<j$, we expand
\enlargethispage{4ex}
\begin{eqnarray*}
\hspace{-0.8em}
  \rho(S_i,S_j)\rho(S_i+S_j,-S_i-S_j)&=&\gamma(S_i)\gamma(S_j)\gamma(S_i+S_j)^{-1}\gamma(S_i+S_j)\gamma(-S_i-S_j)\gamma(\zeroE)^{-1} \hspace{-1.2em}\\
  &=&\gamma(S_i)\gamma(S_j)\gamma(S_k)
\end{eqnarray*}
where $\{i,j,k\}=\{1,2,3\}$. Therefore,
\begin{eqnarray*} 
  \alpha'(\lin)
  &=& \sum_{S\in\lin}{\alpha(S)} - 3 \prod_{S\in\lin}{\gamma(S)} \\
  &=& \alpha(S_1) + \alpha(S_2) + \alpha(S_3) - 3
  \rho(S_1,S_2) \rho(S_3,-S_3). 
\end{eqnarray*}
\end{proof}

\begin{corollary}
\label{cor:p=3}
In the case $p=3$, Theorem~\ref{localpairing} holds with
$\alpha'=\Tr_{M/L'}(\alpha) - 3 N_{M/L'}(\alpha)^{1/3}$ for some
choice of cube root.
\end{corollary}
\begin{proof}
  With notation as in the previous proof we have
  \[ \sum_{S\in\lin}{\alpha(S)} =\Tr_{M/L'}(\alpha)(\lin) \quad
  {\text{and}} \quad
  \prod_{S\in\lin}{\gamma(S)}^3=\prod_{S\in\lin}{\alpha(S)}
  =N_{M/L'}(\alpha)(\lin).\]
\end{proof}

If the only element $x\in L'$ satisfying $x^3=1$ is the element $1$
itself, then Corollary~\ref{cor:p=3} defines $\alpha'$ uniquely in the
case $p=3$. However, if $L'$ contains a non-trivial cube root of
unity, then we must do more to pin down the correct choice of cube
root of $N_{M/L'}(\alpha)$.

\begin{proposition}
\label{cuberoot}
Let $\xi\in H^1(K,E[3])$, and $\alpha\in L^{\mult}$ a representative
for $w_1(\xi)$.
\begin{enumerate}
\item There exist $r \in L^+$ and $s \in L'$ such that
  $N_{L/L^+}(\alpha) = r^3$, $N_{M/L'}(\alpha) = s^3$, $\alpha
  N_{L^+/K}(r) = r N_{M/L}(s)$ and $N_{L/K}(\alpha) = N_{L'/K}(s)$.
\item If $r$ and $s$ are as in (i) then there exists $\rho \in (R
  \otimes_K R)^\mult$, a representative for $w_2(\xi)$ compatible with
  $\alpha$, such that for all $\lin = \{S_1,S_2,S_3\} \in \Lambda$ we
  have $s(\lin) = \rho(S_1,S_2) \rho(S_3,-S_3)$.
\end{enumerate}
\end{proposition}
\begin{proof}
  (i) Let $\rho \in (R \otimes_K R)^\mult$ be a representative for
  $w_2(\xi)$ compatible with $\alpha$, and let $\gamma \in
  \Rbar^\mult$ be as in Definition~\ref{compatible}.

  We put $r(\pm T) = \gamma(T) \gamma(-T)$ and $s(\lin) = \prod_{T \in
    \lin} \gamma(T)$. It is easy to check that $r$ and $s$ are Galois
  equivariant, and so belong to $L^+$ and $L'$.  We compute
\begin{align*}
  N_{L/L^+}(\alpha)(\pm T) &= \alpha(T) \alpha(-T) = \gamma(T)^3
  \gamma(-T)^3 = r(\pm T)^3, \\
  N_{M/L'}(\alpha)(\lin) &= \textstyle\prod_{T \in \lin} \alpha(T) =
  \textstyle\prod_{T \in \lin} \gamma(T)^3 = s(\lin)^3, \\
  \alpha(T) N_{L^+/K}(r) &= \gamma(T)^3 \textstyle\prod_{\zeroE \not=
    P \in E[3]} \gamma(P) = r(\pm T) \textstyle\prod_{T \in \lin}
  s(\lin) = r(\pm T) N_{M/L}(s)(T), \\
  N_{L/K}(\alpha) &= \textstyle\prod_{\zeroE \not= P \in E[3]}
  \gamma(P)^3 = \textstyle\prod_{\lin \in \Lambda} s(\lin) =
  N_{L'/K}(s).
\end{align*}
(ii) If $r$ and $s$ are chosen as in the proof of (i), then
\begin{equation}
\label{nicecuberoot} 
s(\lin) = \rho(S_1,S_2) \rho(S_3,-S_3). 
\end{equation}
for all $\lin = \{S_1,S_2,S_3\} \in \Lambda$.  We must show that this
still holds, for some $\rho$ compatible with $\alpha$, whenever $r$
and $s$ satisfy the conditions in (i).

Let $\widetilde{\Lambda} = \PP(E[3]) \cup \Lambda$ be the set of all
lines in $E[3]$. We write $\Map(E[3],\mu_3)/\mu_3$ for the quotient of
$\Map(E[3],\mu_3)$ by the constant maps. We claim there is an exact
sequence
\begin{equation}
\label{exseq3}
 \frac{ \Map(E[3],\mu_3) }{\mu_3} \to \Map(\widetilde{\Lambda},\mu_3)
\to \frac{ \Map(E[3],\mu_3) }{\mu_3} \times \mu_3 
\end{equation}
where the first map is $\theta \mapsto (\lin \mapsto \prod_{T \in
  \lin} \theta(T))$ and second map is \[\phi \mapsto ( T \mapsto
\textstyle\prod_{ T \in \lin \in \widetilde{\Lambda}} \phi(\lin),
\textstyle\prod_{\lin \in \Lambda} \phi(\lin) ).\] The exactness is
checked by linear algebra over $\F_3$. 

If we change our choices of $r$ and $s$ in (i), then they change by an
element $\phi \in \Map(\widetilde{\Lambda},\mu_3)$. If both choices of
$r$ and $s$ satisfy $\alpha N_{L^+/K}(r)= r N_{M/L}(s)$, then $\phi$
has the property that $\prod_{T \in \lin \in \widetilde{\Lambda}}
\phi(\lin)$ is independent of $T \in E[3]$. If both choices of $s$
satisfy $N_{L/K}(\alpha) = N_{L'/K}(s)$, then $\phi$ has the property
that $\prod_{\lin \in \Lambda} \phi(\lin) = 1$.  So, by the exact
sequence~(\ref{exseq3}), there exists $\theta \in \Map(E[3],\mu_3)$
with $\theta(\zeroE)=1$ and $\phi(\lin) = \prod_{T \in \lin}
\theta(T)$ for all $\lin \in \widetilde{\Lambda}$. It is easy to write
the map
\[ \partial \theta : E[3] \times E[3] \to \mu_3 \, ; \quad
(S,T) \mapsto \theta(S) \theta(T) / \theta(S+T) \]
in terms of $\phi$. Hence, if $\phi$ is Galois equivariant then so is
$\partial \theta$. 

Since $R$ and $L$ are the \'etale algebras of $E[3]$ and $E[3]
\setminus \{\zeroE\}$, we have $R = K \times L$, and there is a
natural inclusion $L^\mult \subset R^\mult$. 
We may then view $w_1$, as defined
in~(\ref{w1}), as a map $w_1:H^1(K,E[3])\rightarrow
R^{\mult}/(R^{\mult})^3$.  It fits in the exact sequence
\begin{equation*}
\label{w_1}
0 \ra E(K)[3] \stackrel{w}{\ra} \mu_3(R)\stackrel{\partial}{\ra} 
(\partial\mu_3(\Rbar))^{G_K} \ra  H^1(K,E[3]) \stackrel{w_1}{\ra} 
R^{\mult}/(R^{\mult})^3 
\end{equation*}
Lemma~\ref{lem:SS} states that $w_1$ is injective.  This means that
$(\partial\mu_3(\Rbar))^{G_K}=\partial(\mu_3(R))$.  Therefore,
multiplying $\theta \in \mu_3(\Rbar)$ by $w(T)$ for some $T \in E[3]$,
we may assume that $\theta \in \mu_3(R)$. In other words, $\theta$
itself and not just $\partial \theta$ is Galois equivariant.  Then,
replacing $\gamma$ and $\rho$ by $\gamma \theta$ and $\rho
\partial \theta$, we see that the conditions of
Definition~\ref{compatible} are still satisfied, but
now~(\ref{nicecuberoot}) holds for the new $s$.
\end{proof}

\begin{corollary}
\label{choiceofroot}
In the case $p=3$, let $r$ and $s$ be as described in
Proposition~\ref{cuberoot}. Then Theorem~\ref{localpairing} holds with
$\alpha'=\Tr_{M/L'}(\alpha) - 3s$.
\end{corollary}

\begin{proof}
  Let $\rho\in(R \otimes_K R)^\mult$ be as described in part (ii) of
  Proposition~\ref{cuberoot}. Then for all
  $\lin=\{S_1,S_2,S_3\}\in\Lambda$, we have $s(\lin)=
  \rho(S_1,S_2)\rho(S_3,-S_3)$.  Using this $\rho$ in Proposition
  \ref{p=3} we get $\alpha'(\lin) = \sum_{i=1}^3 \alpha(S_i) - 3
  \rho(S_1,S_2)\rho(S_3,-S_3) =\Tr_{M/L'}(\alpha)(\lin)-3s(\lin)$.
\end{proof}

\begin{Remark}
  In the case where $[K(E[3]):K]$ is coprime to $3$, Lemma~\ref{field
    extn} allows us to reduce to the case where all the $3$-torsion is
  defined over $K$. If $S,T\in E[3]$ are a basis such that
  $e_3(S,T)=\zeta_3$, then we can choose $\gamma$ in
  Definition~\ref{compatible} such that for $0\leq a,b\leq 2$,
  $\gamma(aS+bT)=\gamma(S)^a\gamma(T)^b$. Consequently, we obtain
  $\iota(\alpha')(T)=\alpha(S)+\alpha(S+T)+\alpha(S-T)-3\gamma(S)\gamma(S+T)\gamma(S-T)=\alpha(S)N_{K(\gamma(T))/K}(1+\gamma(T)+\gamma(T)^2)$.
  Thus, the relevant Hilbert norm residue symbol is
  $\{\alpha(T),\alpha(S)\}$ and, for $n=3$, we recover the formula
  given in \cite{O`Neil} for the period-index obstruction with full
  level $n$-structure.
\end{Remark}

\section{Global computations}
\label{sec:glob}

In this section, $C \subset \PP^2$ will be a smooth plane cubic defined
over a number field $K$. We suppose that $C$ is everywhere locally
soluble.  We write ``$\summ$'' for the isomorphism $\Pic^0(C) \isom
E$, where $E$ is the Jacobian of $C$.  The {\em hyperplane section} of
$C$ (i.e. intersection of $C$ with a line) is a degree $3$ effective
$K$-rational divisor $H$ on $C$, defined up to linear equivalence. If
$H'$ is another degree $3$ effective $K$-rational divisor on $C$, then
the linear system $|H'|$ can be used to define a new embedding $C
\subset \PP^2$ with hyperplane section $H'$.

We are interested in the following problem.

\begin{Problem}
\label{prob:1}
Given a smooth plane cubic $C \subset \PP^2$ with hyperplane section
$H$, and a point $P \in E(K)$, find equations for an embedding $C \to
\PP^2$ whose image is a smooth plane cubic with hyperplane section
$H'$ satisfying $\summ(H' - H) = P$.
\end{Problem}

As described in the proof of \cite[Lemma 1]{SD}, the $K$-rational
effective divisors $H'$ in the required linear equivalence class
correspond to the $K$-rational points on a certain Brauer-Severi
surface $V$. Since $C$ is everywhere locally soluble, so is $V$. By
the Hasse principle for Brauer-Severi varieties we know that $V(K)
\not= \emptyset$, and so $H'$ exists.  Writing down equations for $V$
and then searching for a $K$-rational point is unlikely to be
practical.  We therefore take a different
approach. 

First, we explain how a solution to Problem~\ref{prob:1} helps us
compute the Cassels--Tate pairing.  In Section~\ref{sec:CTpair}, we
take $\zeroE \not= T \in E[3]$ and, after extending our field $K$ so
that $T \in E(K)$, aim to compute $f_T \in K(E)$ with $\divv(f_T) = 3
\aaa_T$ and $\summ(\aaa_T)=T$.  Solving Problem~\ref{prob:1} with
$P=T$ gives us $\aaa_T$ in the form $H'-H$, and from this we can
compute $f_T$.  To say a little about what $f_T$ looks like, we write
$K[x,y,z]_d$ for the space of homogeneous polynomials of degree $d$,
and $\CL(D)$ for the Riemann-Roch space of a divisor $D$.  We also
suppose, for definiteness, that $H = C \cap \{x= 0\}$.  It is known
(see for example \cite[Theorem 7.3.1]{BL}) that for any $d \ge 1$ the
map \[ K[x,y,z]_d \to \CL(dH) \, ; \quad f \mapsto f/x^d \] is
surjective. Taking $d=3$ shows we can write $f_T$ in the form
$f_1/x^3$ where $f_1$ is a ternary cubic meeting $C$ in divisor $3
H'$. By changing our choice of hyperplane section $H$, we could
replace the denominator by the cube of any linear form.

We assume $P \not= \zeroE$ (otherwise Problem~\ref{prob:1} is
trivial).  The curve $C$ may be embedded in $\PP^2$ using either the
linear system $|H|$ or the linear system $|H'|$.  The first of these
gives the embedding we started with. Taking both embeddings together
gives a map $C \to \PP^2 \times \PP^2$.  The image is defined by three
bi-homogeneous forms of degree $(1,1)$. The coefficients may
conveniently be arranged as a $3 \times 3 \times 3$ cube.  These cubes
have many fascinating properties. We first learnt of these from work
of Bhargava and O'Neil (unpublished) and Bhargava and Ho \cite{BH}.
See also \cite{DG}, \cite{HoThesis}, \cite{Ng}.

If we arrange the coefficients of a $3 \times 3 \times 3$ cube into
three $3 \times 3$ matrices, say $M_1, M_2, M_3$, then
\begin{equation}
\label{slice}
F(x,y,z) = \det(x M_1 + y M_2 + z M_3) 
\end{equation}
is a ternary cubic. Since we can slice the cube in three different
directions, this gives us three different ternary cubics. As shown in
\cite[Theorem 1]{Ng}, two of these define the image of $C$ under the
embeddings corresponding to $H$ and $H'$. Moreover an isomorphism
between these two plane cubics is given by the the $2 \times 2$ minors
of the matrix of linear forms in (\ref{slice}). We can then adopt the
point of view in Problem~\ref{prob:1}, namely that we have one curve
with two different embeddings in $\PP^2$.

We are therefore interested in the following problem.

\begin{Problem}
\label{prob:2}
Given a non-singular ternary cubic $F \in K[x,y,z]$, find 
matrices $M_1,M_2,M_3 \in \Mat_3(K)$ satisfying
\[ F(\alpha,\beta,\gamma) = \det(\alpha M_1 + \beta M_2 + \gamma M_3). \]
\end{Problem}

This problem is also considered in \cite{DG}, where an application to coding theory is suggested.

We label the coefficients of $F$ by putting
\begin{align*}
F(x,y,z) & =  a x^3 + b y^3 + c z^3 + a_2 x^2 y + a_3 x^2 z \\
& ~\qquad \quad + \, b_1 x y^2 
+ b_3 y^2 z + c_1 x z^2  + c_2 y z^2 + m x y z.
\end{align*}
By a change of co-ordinates, we may assume $c = F(0,0,1) \not= 0$.  Let
$A_F$ be the free associative $K$-algebra on two indeterminates $x$
and $y$ subject to the relations deriving from the formal identity in
$\alpha$ and $\beta$,
\[ F(\alpha,\beta,\alpha x+ \beta y) = 0.  \] 
Explicitly, $A_F = K\{x,y\}/I$ where $I$ is the ideal generated by the
elements
\begin{align*}
& c x^3 + c_1 x^2 + a_3 x + a,\\
& c (x^2 y + x y x + y x^2) + c_1 (x y + y x) + c_2 x^2 + m x + a_3 y + a_2,\\
& c (x y^2 + y x y + y^2 x) + c_2 (x y + y x) + c_1 y^2 + m y + b_3 x + b_1,\\
& c y^3 + c_2 y^2 + b_3 y + b. 
\end{align*}

In solving Problem~\ref{prob:2}, we are free to multiply $F$ through by
a scalar. If we scale so that $F(0,0,1) = -1$, then without loss of
generality $M_3 = -I_3$.
 
\begin{lemma}
\label{lem:alg}
Let $F \in K[x,y,z]$ be an irreducible ternary cubic with $F(0,0,1)
\not= 0$, and let $M_1, M_2 \in \Mat_3(K)$. The following are
equivalent.
\begin{enumerate}
\item $F(\alpha,\beta,\gamma) = \lambda \det (\alpha M_1 + \beta M_2 -
  \gamma I_3)$ for some $\lambda \in K^\mult$.
\item There is a $K$-algebra homomorphism $A_F \to \Mat_3(K)$ with $x
  \mapsto M_1$ and $y \mapsto M_2$.
\end{enumerate}
\end{lemma}

\begin{proof}
  If (i) holds then $\gamma \mapsto F(\alpha,\beta,\gamma)$ is a
  scalar multiple of the characteristic polynomial of $\alpha M_1 +
  \beta M_2$.  So, by the Cayley-Hamilton theorem,
\begin{equation}
\label{m1m2}
F(\alpha,\beta,\alpha M_1 + \beta M_2) = 0.
\end{equation}
Therefore, $M_1$ and $M_2$ satisfy the relations used to define $A_F$.
This proves (ii). If the minimal polynomial of $\alpha M_1 + \beta
M_2$ has degree $3$ for infinitely many $(\alpha:\beta) \in \PP^1$
then the converse is clear. Otherwise, after replacing $M_1$ and $M_2$
by suitable linear combinations, neither has minimal polynomial of
degree $3$. So $M_1$ and $M_2$ each have an eigenspace of dimension at
least $2$. Since these eigenspaces have non-trivial intersection, it
follows by~(\ref{m1m2}) that $\{F=0\} \subset \PP^2$ contains a line.
This contradicts that $F$ is irreducible.
\end{proof}

We have now reduced Problem~\ref{prob:2} to finding a $K$-algebra
homomorphism $A_F \to \Mat_3(K)$.  Although the connection with
Problem~\ref{prob:2} is new, the algebra $A_F$ was previously studied
by Kuo \cite{Kuo}.  She showed that $A_F$ is an Azumaya algebra of rank
$9$ over its centre $Z(A_F)$, and that $Z(A_F)$ is isomorphic to the
co-ordinate ring of the affine curve $E \setminus \{\zeroE\}$, where
$E$ is the Jacobian of $C = \{ F= 0\} \subset \PP^2$.  In particular
we can specialise $A_F$ at any non-zero point of $E$ to obtain a
central simple algebra of dimension $9$ over the field of definition
of that point.  In fact, Kuo only considered the special case $c=1$ and
$a_3 = b_3 = c_1 = c_2 = 0$, but the general case follows by making
suitable changes of co-ordinates.

We put $r = y (c x^2 + c_1 x + a_3)$, $s = -(c y^2 + c_2 y + b_3)$ and
$t = c x$.  Then, using the support for finitely presented algebras in
Magma \cite{magma}, we were able to check that the centre $Z(A_F)$ is
generated by\footnote{In fact, the elements $\delta_1$ and $\delta_2$
  in the proof of \cite[Lemma 2.1]{Kuo} are equal, and $\xi + m^2/3$
  specialises to $\delta_1 = \delta_2 = \delta/2$}
\begin{align*}
  \xi &= c^2 (x y)^2 - (c y^2 + c_2 y + b_3) (c x^2 + c_1 x + a_3) +
  (c m - c_1 c_2) x y + a_3 b_3
\end{align*}
and 
\begin{align*}
  \eta &= r s t + s t r + t r s
  + a_2 (s t + t s) + b_3 (t r + r t) + c_1 (r s + s r) \\
  & \qquad + (b_3 c_1 - b_1 c) r + (c_1 a_2 - c_2 a) s + (a_2 b_3 -
  a_3 b) t - 6 a b c + a_2 b_3 c_1.
\end{align*}
Moreover, the elements $\xi$ and $\eta$ satisfy
\begin{equation}
\label{jac3}
\eta^2 + A_1 \xi \eta + A_3 \eta = \xi^3 + A_2 \xi^2 + A _4 \xi + A_6, 
\end{equation}
where 
\begin{equation*}
  \label{jac3gen}
  \begin{aligned} \smallskip
    A_1 & =  m, \\ \smallskip
    A_2 & =  -(a_2 c_2+a_3 b_3+b_1 c_1), \\ \smallskip
    A_3 & =  9 a b c - (a b_3 c_2 + b a_3 c_1 + c a_2 b_1)
                 - (a_2 b_3 c_1 + a_3 b_1 c_2), \\
    A_4 & =  -3 (a b c_1 c_2 + a c b_1 b_3  + b c a_2 a_3) \\
        &    \quad + \, a (b_1 c_2^2 + b_3^2 c_1) 
                 + b (a_2 c_1^2 + a_3^2 c_2) 
                 + c (a_2^2 b_3 + a_3 b_1^2) \\ \smallskip
        &    \quad + \, a_2 c_2 a_3 b_3 + b_1 c_1 a_2 c_2 + a_3 b_3 b_1 c_1, \\
    A_6 & =  -27 a^2 b^2 c^2 
               + 9 a b c (a b_3 c_2 + c a_2 b_1 + b a_3 c_1)
               + \ldots +  a b c m^3. 
  \end{aligned}
\end{equation*}
The polynomials $A_i \in \Z[a,b,c, \ldots, m]$ are the coefficients of
the Weierstrass equation for the Jacobian specified in~\cite{ARVT}.
These were obtained by modifying the classical formulae in
\cite{AKM3P}.

Let $\zeroE \not= P = (x_P,y_P) \in E(K)$. Then the specialisation
$A_{F,P}$ of $A_F$ at $P$ is the quotient of $A_F$ by the extra
relations $x_P = \xi$ and $y_P = \eta$. By the work of Kuo cited
above, $A_{F,P}$ is a central simple algebra over $K$ of dimension
$9$. It therefore represents an element in $\Br(K)[3]$.  Kuo also
shows that if $C = \{F=0\} \subset \PP^2$ has a $K$-rational point,
then the Azumaya algebra $A_F$ splits.  By our assumption that $C$ is
everywhere locally soluble, and the local-to-global principle for the
Brauer group, it follows that $A_{F,P} \isom \Mat_3(K)$.  If we can
find such an isomorphism then this immediately gives us a $K$-algebra
homomorphism $A_F \to \Mat_3(K)$ and hence, by Lemma~\ref{lem:alg}, a
solution to Problem~\ref{prob:2}.

The following lemma shows that the point $P$ in the statement of
Problem~\ref{prob:1}, and the point $P$ in the above solution to
Problem~\ref{prob:2} are the same.

\begin{lemma}
  Suppose we solve Problem~\ref{prob:2} by finding an isomorphism
  $A_{F,P} \isom \Mat_3(K)$ for some $\zeroE \not=P \in E(K)$.  Then
  the $3 \times 3 \times 3$ cube we obtain defines a genus one curve
  $C \subset \PP^2 \times \PP^2$ whose projections onto each factor
  are plane cubics with hyperplane sections $H$ and $H'$ satisfying
  $\summ(H-H') = P$.
\end{lemma}
\begin{proof}
  For the proof, we may work over an algebraically closed field, and
  change coordinates so that $C=E$ is an elliptic curve in Weierstrass
  form. Moving the point $P$ to $(x,y)=(0,0)$, we may assume that $E$
  has Weierstrass equation
\begin{equation}
\label{Weqn}
 y^2 + a_3 y = x^3 + a_2 x^2 + a_4 x.
\end{equation}
The image of $(x,y) \mapsto (1:y:x)$ is defined by the ternary cubic
\[ F(x,y,z) = x y^2 + a_3 x^2 y - z^3 - a_2 x z^2 - a_4 x^2 z. \]
Using~(\ref{jac3}) to compute the Jacobian, we recover the Weierstrass
equation~(\ref{Weqn}). Then $A_{F,P} \isom \Mat_3(K)$ via
\[ x \mapsto \begin{pmatrix} 0 & 0 & 1 \\ -a_3 & 0 & 0 \\ -a_4 & 0 &
  -a_2 \end{pmatrix}, \qquad y \mapsto \begin{pmatrix} 0 & 0 & 0 \\ -1
  & 0 & 0 \\ 0 & -1 & 0 \end{pmatrix}. \] 
In particular, we check that $\xi \mapsto 0$ and $\eta \mapsto 0$.  The
images of $x, y$ and $-1$ in $\Mat_3(K)$ form a $3 \times 3 \times 3$
cube. Let $F_i$ be the bi-homogeneous form whose coefficients are
given by the $i$th rows of these matrices, as follows,
\begin{align*}
F_1(x_1,y_1,z_1;x_2,y_2,z_2) &= -z_1 x_2 + x_1 z_2, \\
F_2(x_1,y_1,z_1;x_2,y_2,z_2) &= -a_3 x_1 x_2 - y_1 x_2 - z_1 y_2, \\
F_3(x_1,y_1,z_1;x_2,y_2,z_2) &= -a_4 x_1 x_2 - y_1 y_2 - a_2 x_1 z_2 - z_1 z_2.
\end{align*}
Then $F_1$, $F_2$, $F_3$  define the image of $E \to \PP^2 \times \PP^2$ via
\[ (x,y) \mapsto ((1:y:x), (1:-(y+a_3)/x:x)). \]
Projecting onto each factor gives two embeddings $E \subset \PP^2$ with 
hyperplane sections $H = 3.\zeroE$ and $H' = 2.\zeroE + P$. 
In particular, ${\summ}(H' - H) = P$.
\end{proof}

We have now reduced Problems~\ref{prob:1} and~\ref{prob:2} to the
following problem.

\begin{Problem}
\label{prob:3}
Let $K$ be a number field. Given structure constants for a $K$-algebra
$A$ known to be isomorphic to $\Mat_3(K)$, find such an isomorphism
explicitly.
\end{Problem}

We briefly discuss two algorithms for solving this problem.

\subsection*{Norm equations} By a theorem of Wedderburn (see
\cite[Theorem~2.9.17]{J}), every central simple algebra of dimension
$9$ is a cyclic algebra.  By following the proof (see~\cite{GHPS}
or~\cite{H} for details), Problem~\ref{prob:3} reduces to that of
solving a norm equation for a cyclic cubic extension $L/K$.
Algorithms for solving norm equations do exist (see
\cite[Section~7.5]{CohenGTM193}), but as they involve computing the
class group and units for $L$, they are rarely practical in the
applications of interest to us.

\subsection*{Minimisation and reduction} 
This approach was first suggested by M. Stoll, but with an ad hoc
approach to the reduction. The ``minimisation'' stage is to compute a
maximal order $\OO$ in $A$, using the algorithm in \cite{IR},
\cite{R}. For the ``reduction'' stage we compute trivialisations $A
\otimes_K K_v \isom \Mat_3(K_v)$ for each infinite place $v$, and use
this to embed $\OO$ as a lattice in a Euclidean space of dimension
$\dim_\Q(A)= 9[K:\Q]$.  We then search for a zero-divisor in $A$ by
looking at short vectors in this lattice. Once a zero-divisor is found,
it is easy to find an isomorphism $A \isom \Mat_3(K)$, as described
for example in \cite[Section~5]{GHPS}. In \cite[Paper~III,
Section~6]{descsum} it is shown that if $K=\Q$ then the shortest
vector in the lattice is a zero-divisor. In practice, a zero-divisor
can then be found by the LLL algorithm. In \cite{IRS}, a
complexity-theoretic result is proved describing the behaviour of the
algorithm over a general number field.  The algorithm is only
practical if the discriminant of $K$ is sufficiently small.

\section{Example}
\label{sec:ex}

In this section, we illustrate our work by computing the Cassels--Tate
pairing on the $3$-Selmer group of the elliptic curve 17127b1 in
\cite{CrTables}. This elliptic curve $E/\Q$ has Weierstrass equation
\begin{equation}
\label{weierE}
  y^2 + x y + y = x^3 - x^2 - 19163564 x - 34134737802. 
\end{equation}
The Galois representation $\rho_{E,3} : \Gal(\Qbar/\Q) \to
\GL_2(\Z/3\Z)$ is surjective.  Therefore, the \'etale algebras $L$,
$L'$ and $M$, defined in Section~\ref{sec:CTpair}, are fields. We find
that $L = \Q(u)$ and $L' = \Q(v)$, where $u$ and $v$ are roots of $X^8
- 5 X^6 + 6 X^4 - 3 = 0$ and $X^8 - 6 X^4 + 19 X^2 - 3 = 0$. Moreover,
$M = L(\theta)$ where $\theta^3 = 2 u^6 - 6 u^4 - 3 u^2 + 1$.  The
isomorphism $\iota : L'(\zeta_3) \isom L(\zeta_3)$ and embedding $L'
\subset M$ are given by
\begin{equation}
\label{embs}
\begin{aligned}
  v &\mapsto
  \tfrac{1}{3}(2 \zeta_3 + 1)(u^7 - 4 u^5 + u^3 + 3 u), \\
  v & \mapsto \tfrac{1}{3}(2 u^5 - 7 u^3) \theta^{-1} +
  \tfrac{1}{3}(u^7 - 4 u^5 + u^3 + 3 u).
\end{aligned}
\end{equation}

The bad primes of $E$ are $3$, $11$ and $173$.  Let $\Sset$ the set of
primes of $L$ dividing these primes, and
\[ L(\Sset,3) = \{ x \in L^\mult/(L^\mult)^3 : \ord_\pp(x) \equiv 0
\!\! \pmod{3} \text{ for all } \pp \not\in \Sset \}. \] By
Lemma~\ref{lem:finite}(i), we have
\[S^{(3)}(E/\Q) \subset L(\Sset,3) \cap \Image(w_1).\] We find that
$L(\Sset,3) \cap \Image(w_1) \isom (\Z/3\Z)^3$ is generated by
\begin{align*}
\alpha &= \tfrac{1}{2} (u^7 + u^6 - 4 u^5 - 3 u^4 + 2 u^3 + 1), \\ 
\beta  &= \tfrac{1}{2} (u^7 - 6 u^5 + 10 u^3 + 3 u^2 - 3 u - 5), \\
\gamma &= \tfrac{1}{2} (632 u^7 - 142 u^6 - 2275 u^5
                      + 642 u^4 + 629 u^3 - 720 u^2 + 1059 u - 625),
\end{align*}
and that $S^{(3)}(E/\Q) \isom (\Z/3\Z)^2$ is the subgroup generated by
$\alpha$ and $\gamma$. Moreover, for each of the primes $p=3,11,173$,
we find that $H^1(\Q_p,E[3]) \isom (\Z/3\Z)^2$ is generated by the
images of $\beta$ and $\gamma$.

Let $\alpha' = \Tr_{M/L'}(\alpha) - 3 N_{M/L'}(\alpha)^{1/3}$. Since
$\mu_3 \not\subset L'$, there is no ambiguity in the choice of cube
root. Explicitly,
\[ \alpha' = \tfrac{1}{6} (2 v^7 - 2 v^6 + v^5 
                  - v^4 - 10 v^3 + 19 v^2 + 48 v - 21).  \]

Factoring into prime ideals in $\OO_L$, we find
\begin{align*}
  11 \OO_L &= \pp_1 \pp_2 \pp_3^3, && N \pp_1 = N \pp_2 =11,
  \quad N \pp_3 = 11^2,  \\
  173 \OO_L &= \qq_1 \qq_2 \qq_3 \qq_4 \qq_5, && N \qq_1 = N \qq_2
  =173, \quad N \qq_3 = N \qq_4 = N \qq_5 = 173^2.
\end{align*}
For $p=11$, $173$ we work with the embeddings $L \subset \Q_p$
corresponding to $\pp_1$ and $\qq_1$. In other words, for both $p=11$
and $p=173$, we choose a torsion point $\zeroE \not= T \in E[3]$
defined over $\Q_p$.  By~(\ref{embs}), this also gives an embedding $L'
\subset \Q_p(\zeta_3)$.  Let $\varphi_p = \varphi_{\Q_p}$ be as
defined in Section~\ref{sec:loc}.  Then, up to a global choice of
sign\footnote{This depends on the relationship between the
  embeddings~(\ref{embs}) and the Weil pairing.},
\[\varphi_p(\alpha) = \Ind_{\zeta_3} (\alpha,\alpha')_p,\] where
$\Ind_{\zeta_3}$ is the isomorphism $\mu_3 \isom \frac{1}{3}\Z/\Z$
sending $\zeta_3 \mapsto \frac{1}{3}$, and $(~,~)_p$ is the
$3$-Hilbert norm residue symbol on $\Q_p(\zeta_3)$.  By
Lemma~\ref{lem:finite}(i), Tate local duality and the product
formula~(\ref{prodform}), we have $\varphi_3(\alpha) +
\varphi_{11}(\alpha) + \varphi_{173}(\alpha)=0$. We use this relation
to compute $\varphi_3(\alpha)$ from $\varphi_{11}(\alpha)$ and
$\varphi_{173}(\alpha)$.  Repeating for $\alpha, \beta, \gamma,
\ldots$ we find that $\varphi_p$ takes values:
\[ \begin{array}{c|ccccccc}
p & \alpha & \beta & \gamma & \alpha \beta & \beta \gamma & \alpha \gamma
& \alpha \beta \gamma \\ \hline
3 & 0 & 1 & 0 & 0 & -1 & 0 & 1 \\
11 & 0 & -1 & 0 & 0 & 0 & 0 & 1 \\
173 & 0 & 0 & 0 & 0 & 1 & 0 & 1 
\end{array} \]
We have identified $\frac{1}{3}\Z/\Z \isom \Z/3\Z$ for readability.
The final column is not needed in what follows, but was computed as a
check on our calculations.  Recalling that $\varphi_p$ is a quadratic
form, we can now read off using~(\ref{sym-quad}) that the associated
symmetric bilinear form $[~,~]_p$ takes values: 
\[ 
\begin{array}{c|ccc} [~,~]_3 
& \alpha & \beta & \gamma \\ \hline
\alpha & 0 & -1 & 0 \\
\beta  &-1 & -1 & 1 \\
\gamma & 0 &  1 & 0 \end{array} \quad
\begin{array}{c|ccc} [~,~]_{11} 
& \alpha & \beta & \gamma \\ \hline
\alpha & 0 &  1 & 0 \\
\beta  & 1 &  1 & 1 \\
\gamma & 0 &  1 & 0 \end{array} \quad  
\begin{array}{c|ccc} [~,~]_{173} 
& \alpha & \beta & \gamma \\ \hline
\alpha & 0 &  0 & 0 \\
\beta  & 0 &  0 & 1 \\
\gamma & 0 &  1 & 0 \end{array} 
\]
 
These calculations are in agreement with the fact that, since
$\alpha,\gamma \in S^{(3)}(E/\Q)$, we have $[\alpha,\alpha]_p =
[\alpha,\gamma]_p = [\gamma,\gamma]_p = 0$ for all primes $p$. Since
the local pairing~(\ref{tatepairing}) is non-degenerate, we could also
have predicted in advance that $[\beta,\gamma]_p \not= 0$ for
$p=3,11,173$.

The Selmer group elements $\alpha$, $\gamma$, $\alpha \gamma$,
$\alpha/\gamma$ correspond to plane cubics $C_m$ for $m=1,\ldots,4$.
We used the algorithms in \cite{descsum}, implemented in Magma, to
compute the following equations for $C_m$.
\[ \hspace{-0.3em} \begin{aligned}
    12 x^3 + 7 x^2 y - x^2 z + 20 x y^2 - 99 x y z + 24 x z^2 
  + 43 y^3 + 13 y^2 z - 17 y z^2 + 80 z^3 &= 0 \\
    9 x^3 - 26 x^2 y - 7 x^2 z + 47 x y^2 - 25 x y z + 105 x z^2 
  + 16 y^3 + 47 y^2 z + 27 y z^2 + 54 z^3 &= 0 \\
    x^3 + 2 x^2 y - 15 x^2 z + 40 x y^2 - 11 x y z + 111 x z^2 
    + 8 y^3 + 91 y^2 z + 131 y z^2 + 344 z^3 &= 0 \\
    4 x^3 - 2 x^2 y - x^2 z - 9 x y^2 - 41 x y z + 97 x z^2 
    + 29 y^3 - 23 y^2 z +  257 y z^2 + 282 z^3 &= 0
\end{aligned} \hspace{-0.1em} \]
These equations have been minimised and reduced (see \cite{minred})
and so, in particular, the $C_m$ have the same primes of bad reduction
as $E$.

In each case $m =1,\ldots, 4$, we used the method in
Section~\ref{sec:glob} to compute a ternary cubic $f_m$ with
coefficients in $L$ meeting $C_m$ in 3 non-collinear points each with
multiplicity $3$. In our example, $L$ is a number field, but in general
a similar calculation is necessary over each constituent field of $L$.
The rational function $f_m/x^3$ has divisor $3 H' - 3 H$, where $H$ is
the hyperplane section and $H'$ is another effective divisor of
degree $3$.  Then $[H'-H] \in \Pic^0(C_m) \isom E$ is a non-zero
$3$-torsion point defined over $L$. In our example, there are only two
such points, say $\pm T$. We can switch the sign by replacing $f_m$ by
its $\Gal(L/L^+)$-conjugate. Determining the right choice of sign
takes some care; see Remark~\ref{rem:signs} below.

We scaled each $f_m$ so that (i) the rational function $f_m/x^3$ is as
described in Lemma~\ref{lem:scale-f}, (ii) the coefficients of $f_m$
are in $\OO_L$, and (iii) $f_m$ and the ternary cubic defining $C_m$
are linearly independent mod $\pp$ for all primes $\pp \not\in \Sset$.
In general, it might be necessary to enlarge $\Sset$ to achieve the
last of these conditions. By Lemma~\ref{lem:finite}, the only primes to
contribute to the pairing will be $p=3,11,173$.

The interested reader can find the $f_m$ in the accompanying Magma
file. We have also included the formula for $f_1$ in
Appendix~\ref{app}.

Evaluating each $f_m$ at a $\Q_p$-point\footnote{We were careful to
  choose points that are not $p$-adically close to the zeros of
  $f_m$.}  on $C_m$, we obtained the following elements of
$L_p^\mult/(L^\mult_p)^3$, where $L_p = L \otimes_\Q \Q_p$.
\[ \begin{array}{c|ccccc} p & f_1 & f_2 & f_3 & f_4 & \\ \hline
3 & \gamma & \beta \gamma & \beta \gamma^2 & \beta^2 \gamma^2  \\
11 & \gamma^2 &1 & \gamma & \gamma \\
173 & \beta & \beta^2 \gamma & 1 & \beta^2 \gamma^2  
\end{array} \]
Using the entries in the column headed $f_1$, we compute
\begin{equation}
\label{answer}
\begin{aligned}
  \langle \alpha,\alpha \rangle & = [\gamma,\alpha]_3 +
  [\gamma^2, \alpha]_{11} + [\beta,\alpha]_{173} = 0, \\
  \langle \alpha,\gamma \rangle & = [\gamma,\gamma]_3 + [\gamma^2,
  \gamma]_{11} + [\beta,\gamma]_{173} = 1.
\end{aligned}
\end{equation}
Repeating for $f_2,f_3,f_4$, the Cassels--Tate pairing is given by
\begin{equation}
  \label{ctp:table}
  \begin{array}{c|cccc} \langle~,~\rangle 
    & \alpha & \gamma & \alpha \gamma & \alpha/\gamma \\ \hline
    \alpha & 0 &  1 & 1 & -1 \\
    \gamma  & -1 &  0 & -1 & -1 \\
    \alpha \gamma & -1 &  1 & 0 & 1 \\
    \alpha/\gamma & 1 & 1 & -1 & 0 
\end{array} 
\end{equation}
This is in agreement with the fact that the pairing is bilinear and
alternating. Had we assumed these properties from the outset, it would
only have been necessary to compute one non-zero value of the pairing.
So the only reason for computing more than one of the $f_m$ was to
help check our calculations.

In conclusion, the Cassels--Tate pairing on $S^{(3)}(E/\Q) \isom
(\Z/3\Z)^2$ is non-zero, and hence non-degenerate.  It follows that
$\rank E(\Q)=0$ and the $3$-primary part of $\Sha(E/\Q)$ is
$(\Z/3\Z)^2$.  The first of these facts could more easily be checked
by $2$-descent. The second could have been checked using $9$-descent
(as described in \cite{creutz}), but our method has the advantage of
not requiring any class group and unit calculations beyond those
needed for the $3$-descent.
 
\begin{Remark}
\label{rem:signs}
Replacing $f_m$ by its $\Gal(L/L^+)$-conjugate has the effect of
changing the sign of every entry in the $m$th row
of~(\ref{ctp:table}). We now explain how we made these sign choices in
a consistent way.  We limit ourselves to a few brief details, since
for the applications in the last paragraph we only need that the
pairing is non-zero.

We fix a $3$-torsion point $T \in E(L)$, written in terms of the
Weierstrass equation~(\ref{weierE}).  Each plane cubic $C_m$
corresponds to a pair of inverse elements in $S^{(3)}(E/\Q)$. The
choice of sign could be fixed by specifying an isomorphism
$\Pic^0(C_m) \isom E$ or a covering map $C_m \to E$.  Instead, we
scale the ternary cubic defining $C_m$ so that it has the same
invariants $c_4$ and $c_6$ as~(\ref{weierE}). This scaling is unique
up to sign, and by \cite[Theorem 2.5]{testeqtc} the choice of sign
corresponds to that in $S^{(3)}(E/\Q)$.  By specialising the sign
$\pm$ to $+$ in \cite[Theorem 7.2]{testeqtc}, and using the torsion
point $T$ chosen above, we may scale the equations for the $C_m$ so
that they correspond to $\alpha, \gamma, \alpha \gamma, \alpha/\gamma
\in L^\mult/(L^\mult)^3$, rather than to the inverses of these
elements. Then, when computing the ternary cubic $f_m$ in
Section~\ref{sec:glob}, we work with the algebra $A_{F,P}$, where $F$
is the equation for $C_m$ we just fixed, and $P$ is the image of $T$
under the isomorphism between the elliptic curves~(\ref{weierE})
and~(\ref{jac3}) which when written in the form $x = u^2 x'+ r$, $y =
u^3 y' + u^2 s x' + t$ has $u=+1$.
\end{Remark}

\begin{Remark}
  We have only computed the Cassels--Tate pairing up to a global
  choice of sign. To compute it exactly, we would have to fix a sign
  convention for the Weil pairing and check that the
  embeddings~(\ref{embs}) are compatible with it.  We would also have
  to expand on Remark~\ref{rem:signs}.
\end{Remark}

\newpage

\appendix

\section{Formulae}
\label{app}

The ternary cubic $f_1$ in the example of Section~\ref{sec:ex} is
\begin{align*}  
  f_1
  &=  (10 u^7 - 8 u^6 - 56 u^5 + 42 u^4 + 60 u^3 - 39 u^2 + 36 u - 29) x^3 \\
  &+\tfrac{1}{2} (76 u^7 - 30 u^6 - 321 u^5 + 136 u^4
  + 173 u^3 - 70 u^2 + 213 u - 103) x^2 y \\
  &+\tfrac{1}{2} (-43 u^7 - 24 u^6 + 118 u^5 + 8 u^4 - 106 u^3 +
  67 u^2 + 15 u + 43) x^2 z \\
  &+ \tfrac{1}{2} (135 u^7 - 74 u^6 - 499 u^5 + 260 u^4
  + 145 u^3 - 47 u^2 + 210 u - 118) x y^2 \\
  &+ \tfrac{1}{2} (129 u^7 + 48 u^6 -
  446 u^5 - 75 u^4 + 73 u^3 - 15 u^2 + 237 u - 36) x y z \\
  &+ \tfrac{1}{2} (83 u^7 -
  19 u^6 - 200 u^5 + 192 u^4 - 27 u^3 + 78 u^2 + 54 u - 82) x z^2 \\
  &+ \tfrac{1}{2} (15 u^7 - 40 u^6 - 32 u^5 + 79 u^4 - 87 u^3 + 47 u^2
  - 27 u +
  32) y^3 \\
  &+ \tfrac{1}{2} (-61 u^7 + 46 u^6 + 295 u^5 - 260 u^4 - 299 u^3 +
  149 u^2 - 300 u + 240) y^2 z \\
  &+ \tfrac{1}{2} (-140 u^7 + 84 u^6 + 537 u^5 -
  314 u^4 - 193 u^3 + 32 u^2 - 405 u + 167) y z^2 \\
  &+ \tfrac{1}{2} (-105 u^7 - 86 u^6 + 276 u^5 + 158 u^4 + 26 u^3 - 11
  u^2 - 189 u - 115) z^3.
\end{align*}

\end{document}